\documentclass[12pt]{amsart}
\usepackage{hyperref}

\newtheorem{theorem}{Theorem}[section]

\newtheorem{lemma}[theorem]{Lemma}
\newtheorem{proposition}[theorem]{Proposition}

\theoremstyle{theorem}

\theoremstyle{definition}
\newtheorem{definition}[theorem]{Definition}

\theoremstyle{remark}
\newtheorem{remark}[theorem]{Remark}

\numberwithin{equation}{section}

\newcommand{\der}[1]{\frac{\partial}{\partial #1}}

\begin{document}
\title{Stability and largeness properties of minimal surfaces in higher codimension}
\author{Ailana Fraser}
\address{Department of Mathematics \\
                 University of British Columbia \\
                 Vancouver, BC V6T 1Z2}
\email{afraser@math.ubc.ca}
\author{Richard Schoen}
\address{Department of Mathematics \\
                 University of California \\
                 Irvine, CA 92617}
\email{rschoen@math.uci.edu}
\thanks{
A. Fraser was partially supported by  the 
Natural Sciences and Engineering Research Council of Canada and R. Schoen
was partially supported by NSF grant DMS-2005431. }

\begin{abstract} We consider stable minimal surfaces of genus $1$ in Euclidean space and in Riemannian manifolds.
Under the condition of covering stability (all finite covers are stable) we show that a genus $1$ finite total curvature minimal
surface in $\mathbb R^n$ lies in an even dimensional affine subspace and is holomorphic for some constant orthogonal complex
structure. For stable minimal tori in Riemannian manifolds we give an explicit bound on the systole in terms of a positive lower
bound on the isotropic curvature. As an application we estimate the systole of noncyclic abelian subgroups of the fundamental group
of PIC manifolds. This gives a new proof of the result of \cite{F3} that the fundamental cannot contain a noncyclic free abelian
subgroup. The proofs apply the structure theory of holomorphic vector bundles over genus $1$ Riemann surfaces developed by
M. Atiyah \cite{A}.
\end{abstract}

\maketitle

\section{Introduction}
Applications of minimal surfaces to metric geometry often arise through the second variation, and those surfaces which are the most
rigid are the stable ones; that is, surfaces with positive second variation of area for compactly supported deformations. In favorable
situations, curvature positivity of the ambient manifold is reflected in restrictions on the {\it size} of stable submanifolds. For example,
a positive lower bound on the Ricci curvature limits the length of a stable geodesic and this in turn implies a diameter bound on
such manifolds 

For higher dimensional minimal submanifolds there are very few theorems of this type. In the case of stable minimal surfaces in
three dimensional manifolds of positive scalar curvature it is true that stable two-sided surfaces behave like surfaces of positive 
curvature and as such have diameter bounded in terms of a positive lower bound of the scalar curvature. Generally for higher dimensional
minimal submanifolds there are local as well as global difficulties. The local issues are reflected in the behavior of stable minimal
submanifolds of Euclidean space. Rigidity theorems for complete stable (or volume minimizing) minimal submanifolds are referred to 
as Bernstein-type theorems, and for minimal hypersurfaces they seek to show that global minimizers are hyperplanes. 

For higher codimension submanifolds this cannot be expected since there are natural classes of minimizers which exist in abundance. For
example for a two dimensional stable surface in Euclidean space there are the holomorphic curves with respect to a constant orthogonal  complex structure on an even dimensional Euclidean space. While one might hope that complete stable minimal surfaces are holomorphic,
this question has a long and complicated history, and is in general false. The positive results in this direction are due to Micallef \cite{Mi},
and he proved very general positive results in this direction for surfaces in $\mathbb R^4$ and also extended these to $\mathbb R^n$
under the assumption of genus $0$ and finite total curvature. In Riemannian manifolds there are important results proven in the 
genus $0$ case by Siu-Yau \cite{SiY} and  by Micallef-Moore \cite{MM}. We also mention the very general results which work in
the complex projective space with its natural metric proven by Lawson-Simons \cite{LS}. All of the results except the ones of Micallef in 
$\mathbb R^4$ (and those of \cite{LS}) use
the splitting of holomorphic vector bundles as a direct sum of line bundles over the Riemann sphere due to Birkhoff and Grothendieck.
This result allows one to find sufficiently many holomorphic sections of the complexified normal bundle which then give information
when used as deformations in the complexified stability inequality. 

For surfaces of positive genus there are two key difficulties which arise. First, the splitting theorem is not true, and secondly line bundles
of non-negative degree do not necessarily have holomorphic sections. In the genus $1$ case there is a splitting into indecomposable 
bundles, and a characterization of the indecomposable bundles due to Atiyah \cite{A}. In order to make use of this theory we extend
the notion of stability to that of {\it covering stability} meaning that the surface and all of its finite coverings are stable. The idea is that 
we can get stable surfaces which are arbitrarily large by going to coverings. It should be noted that covering stability is automatically
true for two sided stable minimal hypersurfaces and for holomorphic curves in K\"ahler manifolds.

It was shown by Arezzo-Micallef-Pirola \cite{AMP} that Micallef's genus $0$ theorem fails for genus $2$ stable surfaces in sufficiently high dimensional Euclidean spaces. We do not know if it is true for stable surfaces of finite total curvature of genus $1$, but we are able to 
prove it under the covering stability assumption.

\begin{theorem}
A complete oriented covering stable genus one surface of finite total curvature in $\mathbb{R}^n$ lies in an even dimensional affine subspace and is holomorphic with respect to a constant orthogonal complex structure on that subspace. 
\end{theorem}

In Riemannian manifolds, the major application of stable genus $0$ minimal surfaces is the sphere theorem of Micallef-Moore \cite{MM}.
This again uses the structure of holomorphic vector bundles over the Riemann sphere to give a lower bound on the Morse index of
such surfaces. The first author \cite{F3} extended some of these ideas to the genus $1$ case showing that a sufficiently large covering of
a minimal torus must be unstable. She applied this, together with existence theory for minimizing tori, to show that the fundamental group
of a compact PIC manifold (positive complex sectional curvature on isotropic two planes) cannot contain a free abelian subgroup of
rank greater than $1$. 

Recall that for a non-simply connected compact manifold we can define the {\it systole} to be the minimum length of closed curves
which are not homotopically trivial. A manifold is called $\kappa$-PIC for a constant $\kappa>0$ if all isotropic curvatures are bounded
below by $\kappa$. We prove a quantitative bound on the systole of stable tori in $\kappa$-PIC manifolds.

\begin{theorem} Suppose $N^n$ ($n\geq 4$) is a $\kappa$-PIC manifold for some $\kappa>0$ and suppose
$f:M\to N$ is a stable conformal branched minimal immersion of genus $1$ and let $R$ denote the systole of $M$ in the induced metric. 
There is an absolute constant $C>0$ so that $R\leq C/\sqrt{\kappa}$.
In the general case for $n=4$ or $n\geq 7$ we can take $C=2\pi/\sqrt{3}$, while in exceptional cases which can
occur for $n=5,6$ we can take $C=2(18+\pi)/\sqrt{3}$.
\end{theorem}
It should be noted that there are stable tori in PIC manifolds such as products of $\mathbb S^1$ with spherical space forms of dimension
at least $3$. The product of a long circle with a unit sphere also shows that the systole of a $\kappa$-PIC manifold can be arbitrarily
large. 

We can apply this theorem together with the existence theory for minimizing tori to the geometry of PIC manifolds. First note that
we can define the systole of a subgroup of $\pi_1(N)$ for a compact Riemannian manifold $N^n$ as the minimum length of closed
curves which are freely homotopic to nontrivial curves in the subgroup. We then obtain the following result.

\begin{theorem} Assume that $G$ is an abelian non-cyclic subgroup of $\pi_1(N)$ and that $N$ is a compact $\kappa$-PIC
manifold. The systole $R$ of $G$ is bounded by $C/\sqrt{\kappa}$ for the same constant $C$ as in the previous theorem.
\end{theorem}

Finally we apply this theory to reprove the main theorem of \cite{F3}. The idea is to show that a free abelian group contains a subgroup
with arbitrarily large systole. 

\begin{theorem} (Fraser \cite{F3}) Suppose $N$ is a compact PIC manifold. Then $\pi_1(N)$ cannot contain a free abelian subgroup of rank
greater than $1$.
\end{theorem}

We also note the difficult results of Brendle \cite{B} which yield a complete classification of compact PIC manifolds of dimension
at least $12$ under a mild assumption and also imply the above theorem under those assumptions. That proof involves constructing a Ricci flow with surgeries. This had been done earlier in the case $n=4$ by Hamilton \cite{H} and Chen-Zhu \cite{CZ}.

\section{Complete covering stable surfaces in $\mathbb{R}^n$} \label{section:stability}

Let  $M_0$ be a Riemann surface and $F:  M_0 \rightarrow \mathbb{R}^n$ a branched conformal minimal immersion. $\Sigma=F(M_0)$ is {\em stable} if the second variation of area is nonnegative for every  compactly supported variation of $\Sigma$. 

\begin{definition}
A minimal surface $\Sigma$ is {\em covering stable} if $\Sigma$ as well  as any finite  cover of  $\Sigma$ is stable.
\end{definition}

Let $NM_0$ denote the pullback of the normal bundle of $\Sigma=F(M_0)$ in $\mathbb{R}^n$. The  condition that $\Sigma$ is stable is expressed by the inequality
\[
        \int_{M_0} \|(d s)^\top\|^2 \;  da_0 \leq \int_{M_0} \|(d s )^\perp\|^2 \;da_0
\]
for all compactly supported sections $s \in \Gamma(NM_0)$, where $(\,\cdot\,)^\top$ and $(\,\cdot\,)^\perp$ denote the orthogonal projections onto the tangent and normal space of $\Sigma$ respectively, and $\|\,\cdot\,\|$ denotes the norm and $da_0$ denotes the area form with respect to the  induced metric on $M_0$. 

The index form associated with the second variation of area extends to a Hermitian form on sections of the complexified normal bundle $N_{\mathbb{C}}M_0=NM_0 \otimes \mathbb{C}$. If $z$ is a local complex coordinate on $M_0$, given a section $s \in N_{\mathbb{C}}M_0$ we have $ds=\partial s + \bar{\partial} s$, where $\partial s=(\partial_z s) \, dz$ and $\bar{\partial} s=(\partial_{\bar{z}}s) \, d\bar{z}$, where $\partial_z$  and $\partial_{\bar{z}}$ to denote the (pullback of the) standard differentiation in $\mathbb{R}^n$ in the directions $\partial/\partial z$ and $\partial/\partial \bar{z}$. 
The condition that $\Sigma$ is stable can be re-expressed by the complexified stability inequality (see \cite{Mi}, \cite{MM}):
\begin{equation} \label{equation:stability}
    \int_{M_0}  \|(\partial s)^{\top}\|^2 \;da_0
    \leq \int_{M_0} \|(\bar{\partial} s)^{\perp}\|^2 \;da_0
\end{equation}
for all compactly supported $s \in \Gamma(N_{\mathbb{C}}M_0)$. 

\subsection{Covering stable complete minimal surfaces of finite total curvature in $\mathbb{R}^n$} 

A complete minimal surface $\Sigma$ of finite total curvature  in $\mathbb{R}^n$ is conformally equivalent to a compact Riemann surface $M$ with finitely many punctures (\cite{O}). The Gauss map extends to the compactified surface $M$ as a holomorphic map, and the tangent and normal bundles of $\Sigma$ in $\mathbb{R}^n$ extend with metric and connection to vector bundles $\mathcal{T}$ and $\mathcal{N}$ over the compactified surface $M$ (see \cite{Mi} p. 80). 
Let $E=\mathcal{N} \otimes \mathbb{C}$. There is a unique holomorphic structure on $E$ such that a section $s \in \Gamma(E)$ is holomorphic if $(\partial_{\bar{z}}s)^{\perp}=0$ (\cite{KM}). 

\begin{lemma} \label{lemma:stability-compactified}
Let $\Sigma$ be a complete minimal surface of finite total curvature in $\mathbb{R}^n$ that is covering stable. Then for any compact covering $\pi: \tilde{M} \rightarrow M$ of the compactified surface $M$ we have
\begin{equation} \label{equation:stability-compactified}
     \int_{\tilde{M}} \|(\partial s)^{\top}\|^2 \;da
    \leq \int_{\tilde{M}} \|(\bar{\partial} s)^{\perp}\|^2 \;da
\end{equation}
for any section $s \in \Gamma(\pi^*\mathcal{N} \otimes \mathbb{C})$.
\end{lemma}

\begin{proof}
$\Sigma$ is given by a conformal branched minimal immersion $F: M_0 \rightarrow \mathbb{R}^n$, with $M_0$ conformally equivalent to a compact Riemann surface $M$ with a finite number of points $\{p_1, \ldots, p_l\}$ removed. The induced metric and connection extend to $M$. Consider any compact covering $\pi: \tilde{M} \rightarrow M$, endowed with the pullback metric and connection.
Let $s \in \Gamma(\pi^*\mathcal{N} \otimes \mathbb{C})$, where $\mathcal{N}$ is the extension of the normal bundle from $M_0$ to the compactified surface $M$. To show that $s$ satisfies the inequality  (\ref{equation:stability-compactified}), we use the following logarithmic cut-off function supported away from the puncture points. Choose a coordinate $z$ centered at $p_i$ and for small $\varepsilon$, define $\varphi_i$  by
\[
     \varphi_i(z) =\left\{ \begin{array} {ll} 0 & |z| \le \epsilon^2  \\ 
	                   \frac{\log \left(\frac{|z|}{\epsilon^2}\right)}{-\log \epsilon} & \epsilon^2 \le  |z| \le \epsilon \\ 
	                  1 & \epsilon \le |z|  \end{array} \right.
\]
and define $\varphi$ by
\[
    \varphi=\left\{ \begin{array} {ll} \varphi_i & \mbox{ on } B_\epsilon(p_i),  \; i=1, \ldots, l  \\ 
	                  1 & \mbox{ otherwise}.  \end{array} \right.
\]
Since $\Sigma$ has finite total curvature, $M$ and hence $\tilde{M}$, has quadratic area growth, and we have
\begin{equation} \label{equation:cutoff-limit}
     \lim_{\epsilon \rightarrow 0} \int_{\tilde{M}} \|\nabla \varphi \|^2 \; da =0.
\end{equation}
Then, since $\Sigma$ is covering stable and $\varphi s$ is compactly supported away from the punctures,
\[
    \int_{\tilde{M}} \|(\partial (\varphi s))^{\top}\|^2  \; da
     \leq \int_{\tilde{M}} \|(\bar{\partial} (\varphi s))^\perp\|^2 \; da,
\]
or
\begin{align*}
    \int_{\tilde{M}} \varphi^2  \, \|(\partial s)^{\top}\|^2  \; da
    & \leq \int_{\tilde{M}} \varphi^2 \, \|(\bar{\partial} s)^\perp\|^2 \; da 
        + \int_{\tilde{M}} \|\nabla \varphi\|^2 \, \|s\|^2 \; da \\
        & \quad + 2 \left( \int_{\tilde{M}} \| \nabla \varphi \|^2 \; da \right)^{\frac{1}{2}} 
           \left( \int_{\tilde{M}} \|s\|^2 \| (\bar{\partial} s)^\perp\|^2 \; da \right)^{\frac{1}{2}}.
\end{align*}
Letting $\epsilon \rightarrow 0$ and using (\ref{equation:cutoff-limit}), we obtain (\ref{equation:stability-compactified}).
\end{proof}

\section{Vector bundles over the torus and almost holomorphic sections} \label{section:bundles}

Recall that a holomorphic vector bundle $E$ over a compact complex manifold 
$M$ is {\em indecomposable} if $E$ does not admit a direct sum decomposition $E=E_1 \oplus E_2$ with $E_1$ and $E_2$ proper holomorphic subbundles of $E$. By the Krull-Schmidt theorem, any holomorphic vector bundle  $E$ over a compact complex manifold $M$ admits a direct sum decomposition $E=E_1 \oplus E_2 \oplus \cdots \oplus E_l$ with $E_i$, $1 \leq i \leq l$, indecomposable holomorphic vector bundles, and the decomposition is unique up to reordering of the summands. 
\begin{lemma} \label{lemma:splitting-covers}
Let $E$ be a holomorphic vector bundle over a complex manifold $M$, and let $\cdots \rightarrow M_{k+1} \rightarrow M_k \rightarrow \cdots \rightarrow M_2 \rightarrow M_1=M$ be an infinite tower of finite covers of $M$, $\pi_k: M_k \rightarrow M$, $k=1, 2, \ldots$. Then there exists an integer $K$ such that for all $k>K$, the summands in the indecomposable direct sum decomposition of the lift $\pi_k^*E$ of $E$ to $M_k$ are the lifts of the summands in the indecomposable decomposition of $\pi_K^*E$.
\end{lemma}
\begin{proof}
We have $E=E_1 \oplus E_2 \oplus \cdots \oplus E_l$ with $E_i$, $1 \leq i \leq l$, indecomposable holomorphic vector bundles. 
The direct sum decomposition of $E$ gives a decomposition of the lifted bundle $\pi_2^*E$ over $M_2$,
\begin{equation} \label{equation:cover-decomposition2}
        \pi_2^*E= \pi_2^*E_1 \oplus \pi_2^*E_2 \oplus \cdots \oplus \pi_2^*E_l.
\end{equation}
Each summand in this decomposition may decompose further into indecomposable subbundles of smaller rank.
Successively lifting the indecomposable direct sum decomposition from the previous cover, for $k=2, 3, \ldots$, since $\pi_k^*E$ has finite rank equal to the rank of $E$, the decomposition of the lifted bundle must eventually either fully split into a direct sum of line bundles, or become a stable decomposition into indecomposable subbundles lifted from the previous cover for all $k \geq K$ for some $K$. 
\end{proof}
If $E$ is self-dual, then
\begin{align*}
     E_1 \oplus E_2 \oplus \cdots \oplus E_l \cong (E_1 \oplus E_2 \oplus \cdots \oplus E_l)^*
     \cong E_1^* \oplus E_2^* \oplus \cdots \oplus E_l^*
\end{align*}
and since $c_1(E_i^*)=-c_1(E_i)$, we have
\begin{equation} \label{equation:decomposition}
         E=P \oplus Z \oplus P^*
\end{equation}
where $P=\oplus_i P_i$ is a direct sum of indecomposable bundles of positive degree, $Z=\oplus_j Z_j$ is a direct sum of indecomposable bundles of degree zero, and $P^*=\oplus_i P_i^*$.

We now specialize to the case where $M$ is a compact Riemann surface of genus 1. In this case, we will show that after lifting to a suitable covering, we can find holomorphic and almost holomorphic sections that span the lift of $P \oplus Z$. 
\begin{definition}
If $E$ is Hermitian, given $\epsilon >0$, we say that a section $s \in \Gamma(E)$ is (pointwise) {\em $\epsilon$-almost holomorphic} if  $\|\bar{\partial} s\| \leq \epsilon \|s\|$.
\end{definition}

In \cite{A}, Atiyah classified indecomposable vector bundles over an elliptic curve. We will use the following.

\begin{lemma} \label{lemma:global-generation}
Let $E$ be an indecomposable holomorphic vector bundle of degree $d$ and rank $r$ over a compact Riemann surface $M$ of genus 1. If $d > r+2$, then $E$ is globally generated by holomorphic sections which each have a zero; specifically,  given $p  \in  M$, there are holomorphic sections $s_1, \ldots , s_m \in \Gamma(E)$ which each have a zero such that the fiber $E_p$ of $E$ at $p$ is spanned by $s_1(p), \ldots, s_m(p)$. 
\end{lemma}

\begin{proof}
Let $L$ be a line bundle of degree 2. Then $E \otimes L^*$ has degree $d-2$ and rank $r$. By \cite[Proposition 3.2]{Gu}, since $\deg E \otimes L^* > \mbox{rank}\, E \otimes L^*$, $E \otimes L^*$ is globally generated by holomorphic sections $s_1(p), \ldots, s_m(p)$. By the Riemann-Roch Theorem, a degree two bundle $L$ has two independent sections $t_1$ and $t_2$ which gloabally generate $L$, and $t_1$ and $t_2$ each have two zeros. Therefore, $s_1 \otimes t_1, \ldots, s_m \otimes t_1, s_1 \otimes t_2, \ldots, s_m \otimes t_2$ are holomorphic sections, which each have zeros, that globally generate $E \cong (E \otimes L^*) \otimes L$.
\end{proof}

First consider the indecomposable subbundles of positive degree. If $\pi: \tilde{M} \rightarrow M$ is a covering of degree $k$, then for $k$ sufficiently large the degree $c_1(\pi_k^*P_i)=k c_1(P_i)$ is strictly bigger than the rank of $\pi_k^*P_i$, and so by Lemma \ref{lemma:global-generation}, if $\pi_k^*P_i$ is indecomposable then it is globally generated by holomorphic sections. 

\begin{lemma} \label{lemma:condition2}
Let $E=P \oplus Z \oplus P^*$ be a self-dual holomorphic vector bundle over a compact Riemann surface $M$ of genus 1. Suppose $P$ is globally generated by holomorphic sections which each have at least one zero.
Then $H^0(P^* \otimes (P \oplus Z)^*)=\{0\}$. 
\end{lemma}
\begin{proof}
Let $\alpha \in H^0(P^* \otimes (P \oplus Z)^*))$.  
Given any holomorphic section $s \in H^0(P)$ with a zero, $\alpha(s, \cdot)$ is a holomorphic section of $(P \oplus Z)^*$ with a  zero. By \cite[Theorem 5]{A} an indecomposable holomorphic vector bundle 
of degree zero has either no holomorphic sections or a unique holomorphic section with no zeros. Therefore $Z^*$ has no holomorphic sections with a  zero. But $P^*$ has no holomorphic sections, and so $(P \oplus Z)^*$ has no homomorphic sections with a  zero. Therefore, $\alpha(s, \cdot)=0$ for all holomorphic sections $s$ of $P$ with a zero. Since $P$ is spanned by holomorphic sections which each have a zero, $\alpha =0$.
\end{proof}

\subsection{Degree zero indecomposable bundles}

We consider here the case of an indecomposable summand of degree zero. Note that such bundles are topologically trivial, 
but may not have holomorphic sections. We show that after lifting to a suitable covering torus we can find $\epsilon$-almost holomorphic sections which span the lifted bundle for any $\epsilon>0$. We begin the discussion with the case of a degree zero line bundle $L$. 

We choose a lattice generated by $\{1,\tau\}$ where $\tau$ is a complex number in the upper half plane, and we assume
that our torus is $M=\mathbb C/\Lambda$ where $\Lambda$ is the lattice subgroup $\Lambda=\{m+n\tau:\ m,n\in\mathbb Z\}$.
For a positive integer $k$ we consider the sub-lattice $k\Lambda\subset \Lambda$ and we denote by $M_k$ the
corresponding covering torus $M_k=\mathbb C/k\Lambda$ which is a covering torus of degree $k^2$. 

Now suppose we have a degree zero line bundle $L$ as a summand of our bundle $E$. It is well known that $L$ admits a flat $U(1)$ connection $\nabla$.  The lift of $L$ to the universal covering $\mathbb C$ is a trivial bundle and has
a global unit parallel section $s_0$. The action of the deck group is then generated by $\phi,\theta\in [0,2\pi)$ where
$s_0(z+1)=e^{i\phi} s_0(z)$ and $s_0(z+\tau)=e^{i\theta}s_0(z)$. We now have the following lemma.
\begin{lemma} \label{flatline} Given any $\epsilon>0$ there exists a positive integer $k_0$ so that for $k\geq k_0$ there is a unit length
section $s$ of the lift $L_k$ of $L$ to $M_k$ satisfying $\|\partial s/\partial \bar{z}\|\leq \epsilon$.
\end{lemma} 
\begin{proof} We observe that $s_0(z+k)=e^{i\phi_k}s_0(z)$ and $s_0(z+k\tau)=e^{i\theta_k}s_0(z)$ where 
$\phi_k=k\phi \mod 2\pi$ and $\theta_k=k\theta \mod 2\pi$. By a linear change of coordinates we can express 
$z=\xi+\eta\tau$ where $z=x+iy$, $\tau=\tau_1+i\tau_2$, $\xi=x-y\frac{\tau_1}{\tau_2}$, and $\eta=\frac{y}{\tau_2}$. 
We now define the section 
\[ 
    s(z)=e^{-i \, \left(\frac{\xi}{k}\phi_k+\frac{\eta}{k}\theta_k\right)}s_0(z).
\]
We note that $z+k=(\xi+k)+\eta\tau$ and $z+k\tau=\xi+(\eta+k)\tau$, and therefore we have 
$s(z+k)=s(z)$ and $s(z+k\tau)=s(z)$. Thus $s$ defines a unit length section of $L_k$ over $M_k$.

We compute 
\[
     \frac{\partial s}{\partial \bar{z}}\equiv \nabla_{\frac{\partial}{\partial \bar{z}}} s
     =-\frac{i}{k} \left( \frac{\partial \xi}{\partial\bar{z}} \ \phi_k  + \frac{\partial \eta}{\partial\bar{z}} \ \theta_k \right) s. 
\]     
By direct calculation we have 
\[
       \frac{\partial \xi}{\partial\bar{z}} =\frac{1}{2}\left(1-i \frac{\tau_1}{\tau_2} \right), \qquad  
       \frac{\partial \eta}{\partial\bar{z}}=\frac{i}{2\tau_2}. 
\]
Thus we see that $\|\partial s/\partial\bar{z}\|\leq c/k$ where the constant
$c$ depends only on the lattice $\Lambda$. The result now follows with $k_0$ a positive integer chosen so that $k_0\geq c/\epsilon$
\end{proof}

Now if $E$ is an indecomposable bundle of rank $r>1$, we use the work of M. Atiyah \cite[Theorem 5]{A} which asserts that any
such bundle is the tensor product of a degree zero line bundle $L$ with a unique bundle $F_r$ called the Atiyah
bundle. The bundle $F_r$ is uniquely characterized by the conditions that it is of degree zero, indecomposable of rank $r$, 
and has a non-zero holomorphic section. It is shown in \cite[Corollary 2]{A} that there is a filtration of holomorphic sub-bundles
\[ F_1\subset F_2\subset\ldots\subset F_r
\]
such that the rank of $F_i$ is $i$ and such that $F_1$ and $F_i/F_{i-1}$ are trivial line bundles for $i=2,\ldots, r$.

For our purposes we choose a particular presentation of $F_r$ given by a representation $\rho$ of the fundamental group
$\Lambda$ into $GL(r,\mathbb C)$ given by $\rho(1)=I$ and $\rho(\tau)=A_\delta$ for a positive number $\delta$. Where
$A_\delta$ is the upper triangular matrix with diagonal entries $1$, first super-diagonal entries $\delta$, and all other
entries $0$. The bundle $F_r$ is then the holomorphic bundle determined by $\rho$; that is 
\[ F_r=(\mathbb C\times \mathbb C^r)/\Lambda
\]
where $\Lambda$ acts on the product by $(z,w)\to (z+\lambda,\rho(\lambda)w)$ for $\lambda\in \Lambda$. We see from
the construction that $F_r$ has a flat connection $\nabla$ with holonomy generated by $A_\delta$. If we denote the
standard basis of $\mathbb C^r$ by $e_1,\ldots, e_r$ we see that $e_1$ is invariant under the holonomy and therefore
defines a holomorphic section of $F_r$. It is easily seen that the bundle is indecomposable of degree zero, so the
bundle is holomorphically isomorphic to $F_r$ for any $\delta>0$.  Since there is no invariant metric under $\rho$,
and we seek $\epsilon$-almost holomorphic sections with respect to a chosen metric, we normalize our trivialization
on $\mathbb C$ so that $e_1(0),e_2(0),\ldots,e_r(0)$ form a unitary basis compatible with the filtration so that the
first $j$ span $F_j$ at $z=0$. Observe that for any $\delta\neq 0$ the representation is conjugate to that with $\delta=1$
by the Jordan form. However, this conjugation must be done in the larger group of invertible upper triangular matrices
and does not preserve the metric normalization. On the other hand we see that the flat connection $\nabla$ is
independent of $\delta$, and thus it follows that there is a constant $c$ independent of $\delta$ such that the
eigenvalues of the Hermitian matrix $\langle e_j,e_k\rangle$ are bounded above by $c$ and below by $c^{-1}$
on the disk $\{|z|<1+|\tau|\}$. If a bundle with metric and compatible connection is isomorphic to $F_r$ we show that
it can be spanned by sections which are arbitrarily close to being holomorphic, so that it is almost trivial. 

Now if we consider any degree zero indecomposable bundle $E$ of rank $r$, it is given by $E=L\otimes F_r$ and we 
can also find a flat connection in a natural way by combining the $U(1)$ connection of $L$ with that of $F_r$. We take
our parallel trivialization on $\mathbb C$ as $v_1,\ldots, v_r$ where $v_j=s_0\otimes e_j$ and we see that the
holonomy representation $\rho$ is determined by $\rho(1)=e^{i\phi}I$ and $\rho(\tau)=e^{i\theta}A_\delta$ for some
$\phi,\theta\in [0,2\pi)$. We also have the corresponding filtration $E_1\subset E_2\subset\ldots\subset E_r$ where $E_j=L\otimes F_j$.We now state the first main theorem of this section.  

\begin{theorem} \label{theorem:sections}
Suppose $E$ is an indecomposable bundle of degree zero over $M$ with a Hermitian metric compatible 
with the holomorphic structure.  Given any $\epsilon>0$ there exists a positive integer $k_0$ so  that for $k\geq k_0$ 
there is a basis of sections $s_1,\ldots,s_r$ of the lift $E_k$ of $E$ to $M_k$ which respects the filtration and satisfies
$\|\partial s_j/\partial\bar{z}\|\leq \epsilon \|s_j\|$. Furthermore the matrix $\langle s_j,s_k\rangle$ has eigenvalues bounded
above and below by positive multiplicative constants independent of $k$.
\end{theorem} 

\begin{proof} First note that since $E$ is isomorphic to $L\otimes F_r$ if we choose a metric on $F_r$ and $L$
we get a corresponding metric $\langle\cdot,\cdot\rangle_0$  on the tensor product. Since $M$ is compact this metric
is equivalent up to constants with the given metric $\langle\cdot,\cdot\rangle$.

We let $s$ be the unit length section of $L_k$ constructed in the proof of the lemma with 
$\|\partial s/\partial\bar{z}\|\leq c/k$. We will fix $k$ sufficiently large and construct the sections $s_j=s\otimes w_j$
where $w_j$ is a section invariant under the full lattice $\Lambda$ and almost holomorphic. Thus $w_j$
is the lift of a section of $F_r$ over $M$ whose construction we now describe.

We consider the lattice $\Lambda$ and show how to construct almost holomorphic sections of $F_r$ over this lattice (no
covering is necessary). We let $N=A_\delta-I$ and define the matrix logarithm of $A_\delta$ by
\[ B=\log(I+N)=\sum_{j=1}^{r-1}(-1)^{j+1}\frac{N^j}{j}
\]
which is a finite sum because $N$ is a nilpotent matrix. Now we have the normalized parallel sections $e_j$ 
over $\mathbb C$ and we define $w_j$
by 
\[ w_j(z)=e^{-\eta B}e_j(z).
\]
We then have $w_j(z+1)=w_j(z)$ and $w_j(z+\tau)=e^{-\eta B}A_\delta^{-1} e_j(z+\tau)=w_j(z)$. We clearly have
$\|B\|\leq c\delta$. We now choose our metric $\langle\cdot,\cdot\rangle_0$ so that the $w_j$ are orthonormal. Note that
this metric depends on $\delta$, but converges as $\delta$ goes to $0$ to the metric on the trivial bundle with trivialization $e_1, \ldots, e_r$. Note that
the holomorphic structure does not converge in this limit. Therefore we have
\[ \frac{\partial w_j}{\partial\bar{z}}=-\frac{\partial\eta}{\partial\bar{z}}Bw_j=-\frac{i}{2\tau_2}Bw_j,
\]
and thus $\|\partial w_j/\partial\bar{z}\|_0\leq c\delta\|w_j\|_0$.

To complete the proof we let $s_j=s\otimes w_j$ where $w_j$ is lifted from $M$, and we have
\[ 
          \left\| \frac{\partial s_j}{\partial\bar{z}} \right\|_0
          =\left\| \frac{\partial s}{\partial\bar{z}} \otimes w_j 
              +s \otimes \frac{\partial w_j}{\partial\bar{z}} \right\|_0
          \leq c\left(\frac{1}{k}+\delta\right)\|s_j\|_0.
\]
We may now fix $k_0$ large enough that $c/k_0\leq \epsilon/2$ and then choose $\delta$ small enough that 
$c\delta\leq\epsilon/2$ and we get $\|\partial s_j/\partial\bar{z}\|\leq \epsilon\|s_j\|$ as claimed. 

Finally we observe
that the matrix $\langle s_j,s_k\rangle_0$ is equal to $\langle w_j,w_k\rangle_0$ because $s$ has unit length. Since this
matrix is invariant under the lattice $\Lambda$ its eigenvalues are bounded above and below independent of $k$. Finally
since the metrics $\langle\cdot,\rangle_0$ and $\langle\cdot,\cdot\rangle$ are equivalent over $M$ the eigenvalues are
bounded above and below for the given metric.

\end{proof}

Now we will use the covering stability assumption which implies that for any positive integer $k$ and any section
of the lifted bundle $E$  to $M_k$ we have for any section $s$ of $E$ the inequality
\[ \int_{M_k} \|(\partial s)^\top\|^2\ da\leq \int_{M_k}\|(\bar{\partial}s)^\perp\|^2\ da.
\]
We now prove our second main theorem of this section.
\begin{theorem} \label{theorem:degree-zero}
Suppose $E$ is an indecomposable degree zero sub-bundle of the extended complexified normal bundle of $M$.
If $M$ is covering stable it follows that the second fundamental form term $(\partial s)^\top=0$ for any section $s$
of $E$.
\end{theorem}

\begin{proof} We begin with the degree zero line sub-bundle $E_1$ and observe that if we choose a unit section $s_0$
of $E_1$ it follows that if $(\partial s_0)^\top$ is not identically zero then there is a constant $c>0$ so that
\[ \int_{M_k}\|(\partial s_0)^\top\|^2\ da=k^2\int_M \|(\partial s_0)^\top\|^2\ da\geq ck^2
\]
where we lift $s_0$ to a section over $M_k$. Let $s_1$ be the almost holomorphic section of the lift of $E_1$
over $M_k$ from Theorem \ref{theorem:sections} and note that for $k$ large we have 
\[ \int_{M_k}\|(\bar{\partial}s_1)^\perp\|^2\ da\leq \epsilon \int_{M_k}\|s_1\|^2\ da\leq c\epsilon k^2.
\]
By the covering stability condition it follows that
\[  \int_{M_k}\|(\partial s_1)^\top\|^2\ da\leq c\epsilon k^2.
\]
Since $E_1$ is a line bundle and $s_1$ has length bounded from below independent of $k$ it follows that
\[ \int_{M_k}\|(\partial s_1)^\top\|^2\ da\geq ck^2
\]
which is a contradiction for small enough $\epsilon$. Therefore we conclude that $(\partial s)^\top=0$ for sections
$s$ of $E_1$.

We now proceed step by step through the filtration. Assume that for some $j>1$ we have $(\partial s)^\top=0$ 
for sections $s$ of $E_{j-1}$.  Let $s_0$ be a unit section of $E_j$ which is orthogonal to $E_{j-1}$. If $(\partial s)^\top$ is not identically zero for sections $s$ of $E_j$ then there is a constant $c>0$ so that
\[ \int_{M_k}\|(\partial s_0)^\top\|^2\ da\geq ck^2
\]
where we lift $s_0$ to a section over $M_k$. We now use the stability assumption for the almost holomorphic section
$s_j$ to conclude that 
\[ \int_{M_k}\|(\partial s_j)^\top\|^2\ da\leq c\epsilon k^2.
\]
Now if we write $s_j=\alpha s_0+\sum_{l=1}^{j-1}c_ls_l$ for complex numbers $\alpha$ and $c_l$, it follows from the 
bound on the metric that 
\[ 
     |\alpha|^2=\|s_j-\sum_{l=1}^{j-1}c_ls_l\|^2
     \geq c_0 \|s_j-\sum_{l=1}^{j-1}c_ls_l\|_0^2
     \geq c(1+\sum_{l=1}^{j-1}|c_l|^2)
     \geq c
\]
for positive constants $c$, $c_0$, where in the first inequality we used that all the metrics are uniformly equivalent for small $\delta$. It follows that $\|(\partial s_j)^\top\|^2=|\alpha|^2\|(\partial s_0)^\top\|^2$ since the second
fundamental form vanishes on $E_{j-1}$ and so 
\[  \int_{M_k}\|(\partial s_j)^\top\|^2\ da\geq ck^2.
\]
This contradiction shows that if $(\partial s)^\top=0$ for sections $s$ of $E_{j-1}$ then the same is true for sections of
$E_j$. This shows that the second fundamental form vanishes on $E$ and completes the proof.
\end{proof}

\section{Covering stable surfaces of genus one in $\mathbb{R}^n$}

In this section we assume that $\Sigma$ is a complete minimal surface of finite total curvature given by a conformal branched minimal immersion $F: M_0 \rightarrow \mathbb{R}^n$. $M_0$ conformally equivalent to a compact Riemann surface $M$ %of {\em genus one} 
with a finite number of points removed. We let $E=\mathcal{N} \otimes \mathbb{C}$ where $\mathcal{N}$ is the extension of the normal bundle of $M_0$ to the compactified surface $M$. Let $(\cdot,\cdot)$ denote the complex bilinear extension of the metric on $\mathcal{N}$ to $E$. There is a unique holomorphic structure on $E$ such that a section $s \in \Gamma(E)$ is holomorphic if $(\partial_{\bar{z}}s)^{\perp}=0$ (\cite{KM}). 
Since $E$ is a self-dual holomorphic vector bundle over the compact surface $M$, $E$ admits a direct sum decomposition $E=P \oplus Z \oplus N$ where $P=\oplus_i P_i $, $Z=\oplus_j Z_j$ and  $N=\oplus_i P_i^*$ with $P_i$, $Z_j$ indecomposable positive and zero subbundles.  The following propositions gives conditions which imply that $\Sigma$ is holomorphic.

\begin{proposition} \label{proposition:holomorphic}
Let $\Sigma$ be a complete minimal surface in $\mathbb{R}^n$ of finite total curvature. Suppose that:
\begin{enumerate}
\item[(i)]
For all $s \in \Gamma(P \oplus Z)$ we have $(\partial s)^\top=0$
\item[(ii)]
$H^0(P^* \otimes (P \oplus Z)^*)=\{0\}$
\end{enumerate}
If $\Sigma$ lies fully in $\mathbb{R}^n$, then $Z=\{0\}$ and $\Sigma$ is holomorphic with respect to an orthogonal complex structure on $\mathbb{R}^n$.
\end{proposition}

\begin{proof}
We divide the proof into the following steps.
\\

\noindent
\underline{\em Step 1.} $P^\perp=P \oplus Z$, where $P^\perp$ denotes the orthogonal complement of $P$ in $E$ with respect to the complex bilinear pairing $(\cdot,\cdot)$. \\

Define $A: \Gamma(P) \times \Gamma(P \oplus Z) \rightarrow C^\infty(M,\mathbb{C})$ by $A(s,t)=(s,t)$. Then,
\[
     \der{\bar{z}}(A(s,t))=\der{\bar{z}}(s,t)=(\partial_{\bar{z}}^\perp s,t)+(s,\partial_{\bar{z}}^\perp t)
\]
but,
\begin{align*}
    \der{\bar{z}}(A(s,t))&=(\partial_{\bar{z}}^\perp A)(s,t) +A(\partial_{\bar{z}}^\perp s, t)+A(s, \partial_{\bar{z}}^\perp t) \\
    &=(\partial_{\bar{z}}^\perp A)(s,t) +(\partial_{\bar{z}}^\perp s, t)+(s, \partial_{\bar{z}}^\perp t)
\end{align*}
and so $(\partial_{\bar{z}}^\perp A)(s,t)=0$ for all $s \in  \Gamma(P)$, $t \in \Gamma(P \oplus Z)$. Therefore, $A \in H^0(P^* \otimes (P \oplus Z)^*)$, and by assumption (ii), $A \equiv 0$. It follows that $P \oplus Z \subset P^\perp$.
But $\dim P^\perp=\dim E - \dim P=\dim (Z \oplus P^*)=\dim (Z \oplus P)$, and so we must have $P^\perp=P\oplus Z$.
\\

\noindent
\underline{\em Step 2.} $P|_{M_0} \oplus Z|_{M_0} \oplus T^{1,0}_{\mathbb{C}}M_0=M_0 \times \Lambda$, where $\Lambda$ is a constant subspace of $\mathbb{C}^n$. \\

Define $B: \Gamma(P) \times \Gamma(P \oplus Z) \rightarrow C^\infty(M,\mathbb{C})$ by $B(s,t)=(\partial_z t, s)$. Using Step 1, note that  $B$ is $C^\infty(M,\mathbb{C})$-bilinear and defines a section of $P^* \otimes (P\oplus Z)^*$. We claim that $B \in H^0(P^*  \otimes (P \oplus Z)^*)$. To see this, from the definition of $B$ we have
\begin{align} \label{equation:B-deriv}
      \der{\bar{z}} (B(s,t))
      &=(\partial_{\bar{z}} \partial_z t, s) + (\partial_z t, \partial_{\bar{z}} s).
\end{align}
On the other hand, 
\begin{align*}
      \der{\bar{z}} (B(s,t))
      &=(\partial_{\bar{z}}^\perp B)(s,t) +B(\partial_{\bar{z}}^\perp s,t) +B(s,\partial_{\bar{z}}^\perp t) \\
      &= (\partial_{\bar{z}}^\perp B)(s,t) +(\partial_z t,\partial_{\bar{z}}^\perp s) 
      +(\partial_z  \partial_{\bar{z}}^\perp t, s).
\end{align*}
We have
\[
    \partial_{\bar{z}}^\perp  t = \partial_{\bar{z}} t -  \partial^\top_{\bar{z}} t.
\]
On $M_0$, $\partial^\top_{\bar{z}} t = \alpha F_z + \beta F_{\bar{z}}$, for some smooth functions $\alpha$, $\beta$, and 
\[
   (\partial_z \partial^\top_{\bar{z}} t,s) 
   = (\alpha_z F_z +\beta_z F_{\bar{z}} + \alpha F_{zz} + \beta F_{\bar{z} z}, s)
   = (\alpha F_{zz}, s) = -  \alpha (F_z, \partial^\top_z s)
   =0
\]
by assumption (i). Therefore,
\begin{align*}
      \der{\bar{z}} (B(s,t))
      &=(\partial_{\bar{z}}^\perp B)(s,t) +(\partial_z t,\partial_{\bar{z}} s) 
      +(\partial_z  \partial_{\bar{z}} t, s) \\
      &=(\partial_{\bar{z}}^\perp B)(s,t) +(\partial_z t,\partial_{\bar{z}} s) 
      +(\partial_{\bar{z}}  \partial_z t, s),      
\end{align*}
which together with (\ref{equation:B-deriv}) implies that  $\partial_{\bar{z}}^\perp B=0$ on $M_0$. It follows that $B \in H^0(P^*  \otimes (P \oplus Z)^*)$, as claimed. By assumption (ii), $B \equiv 0$.

Let $\xi=P|_{M_0} \oplus Z|_{M_0} \oplus T^{1,0}_{\mathbb{C}}M_0$. We now show that $\xi$ is closed under $d$, $d: \Gamma(\xi) \rightarrow \Gamma( \xi  \otimes T^*M_0)$. We have $(dF_z , F_z)=\frac{1}{2}d(F_z, F_z)=0$ since $F$ is conformal, and $(dF_z  ,s)=-(F_z  , \, \partial s)=0$ for all $s \in \Gamma(P \oplus Z)$ by the minimality of $F$ and assumption (i). If $t \in \Gamma(P \oplus Z)$, then $(\partial t , s)=0$ for all $s \in \Gamma(P)$ since $B \equiv 0$ from above. Therefore, $\partial t \in \Gamma(P^\perp \otimes T^*M_0)=\Gamma((P \oplus Z) \otimes T^*M_0)$, by {\em Step 1}. If $t \in \Gamma(P \oplus Z)$, then $(\bar{\partial} t , F_z)=-( t \, d\bar{z} , F_{z\bar{z}})=0$ by the minimality of $F$, and so $(\bar{\partial} t)^\top \in \Gamma (T^{1,0}_\mathbb{C}M_0 \otimes T^*M_0)$. Finally, $P \oplus Z$ is preserved under $\bar{\partial}^\perp$ since it is a holomorphic sub-bundle of $E$. Therefore, $\xi$ is closed under $d$, and $\xi=M_0 \times \Lambda$ where $\Lambda$ is a subspace of $\mathbb{C}^n$.
\\

\noindent
\underline{\em Step 3.} If $\Sigma$ lies fully in $\mathbb{R}^n$, then $Z=\{0\}$.  \\

Observe that $\Lambda \cap \bar{\Lambda}  = (P_q \oplus Z_q) \cap (\bar{P}_q \oplus \bar{Z}_q)$ for any $q \in M_0$. It follows from {\em Step 1} that $(s_1 , s_2)=0$ for all $s_1, \, s_2 \in \Gamma(P)$, and so $P$ is orthogonal to $\bar{P}$ with respect to the Hermitian inner product. Therefore, $\Lambda \cap \bar{\Lambda}=Z_q  \cap \bar{Z}_q$. Let $p=\mbox{rank}\,P$ and $z=\mbox{rank} \,Z$. Note that $\Lambda+ \bar{\Lambda}=\mathbb{C}^n$ and so
\begin{align*}
     \dim (Z_q \cap \bar{Z}_q) &= \dim (\Lambda \cap \bar{\Lambda}) 
     =\dim \Lambda + \dim \bar{\Lambda} - \dim (\Lambda + \bar{\Lambda}) \\
     &= (p+z+1)+(p+z+1) -n = 2p+2z-(n-2) \\
     &=z
\end{align*}
where in the last equality we used that $n-2=2p+z$ since $E=P \oplus Z \oplus P^*$. Therefore, $Z=\bar{Z}$, and $Z$ is the complexification of a sub-bundle $M_0 \times W$ of the normal bundle $NM_0$, where  $W$ is a constant subspace of $\mathbb{R}^n$. This means that $\Sigma=F(M_0)$ lies in an affine subspace of $\mathbb{R}^n$ orthogonal to $W$. Since $\Sigma$  lies fully in $\mathbb{R}^n$, $Z=\{0\}$.
\\

We may define complex structure $J$ on $\mathbb{R}^n$ by $J=iI$ on $P|_{M_0} \oplus T^{1,0}_{\mathbb{C}}M_0$ and $J=-iI$ on $\bar{P} \oplus T^{0,1}_{\mathbb{C}}M_0$. Since $P|_{M_0} \oplus T^{1,0}_{\mathbb{C}}M_0$ and $\bar{P} \oplus T^{0,1}_{\mathbb{C}}M_0$ are orthogonal with respect to the Hermitian inner product, this defines a complex structure, and $\Sigma$ is $J$-holomorphic (see \cite[Theorem A]{Mi}).
\end{proof}

We now prove the main theorem.

\begin{theorem}
A complete oriented covering stable genus one surface $\Sigma$ of finite total curvature in $\mathbb{R}^n$ lies in an even dimensional affine subspace and is holomorphic with respect to an orthogonal complex structure on that subspace.
\end{theorem}

\begin{proof}
Restricting to a subspace if necessary, we may assume that $\Sigma$ lies fully in $\mathbb{R}^n$.
The extended complexified normal bundle $E$ is a self-dual holomorphic vector bundle over a compact surface $M$ of genus one. Consider the tower of covers $\cdots \rightarrow M_{k+1} \rightarrow M_k \rightarrow \cdots \rightarrow M_2 \rightarrow M_1=M$, where $\pi_k: M_k \rightarrow M$ is the covering of $M$ corresponding to the subgroup $2^k\mathbb{Z} \oplus 2^k \mathbb{Z}$ of the fundamental group of $M$, and let $\Sigma_k$ denote the corresponding cover of $\Sigma$. 

By Lemma \ref{lemma:splitting-covers}, there exists $K$ such that for all $k > K$, the summands in the indecomposable direct sum decomposition of the lift $\pi_k^*E$ of $E$ to $M_k$ are the lifts of the summands in the indecomposable decomposition of $\pi^*_KE$. For all $k >K$, the lift of an indecomposable positive subbundle of $\pi_K^*E$ is indecomposable, and by section \ref{section:bundles}, we may choose $k$ sufficiently large such that it is globally generated by holomorphic sections which each have at least one zero. Suppose $\pi^*_kE=P \oplus Z \oplus P^*$ where $P$ is the direct sum of the indecomposable holomorphic subbundles of positive degree and $Z$ is the direct sum of the indecomposable holomorphic subbundles of degree zero. Then $P$ is globally generated by holomorphic sections which each have at least one zero.
Since $\Sigma$ is covering stable, by Lemma \ref{lemma:stability-compactified}, if $s$ is a holomorphic section of $P$, then
\begin{equation*} 
     \int_{M_k} \|(\partial s)^{\top}\|^2 \;da
    \leq \int_{M_k} \|(\bar{\partial} s)^{\perp}\|^2 \;da=0,
\end{equation*}
and so $(\partial s)^{\top}=0$. Since $P$ is globally generated by holomorphic sections, $(\partial s)^{\top}=0$ for all sections $s$ of $P$.  

By Theorem \ref{theorem:degree-zero} the second fundamental form $(\partial s)^\top=0$ for any section $s$
of $Z$. 

Therefore, we have $(\partial s)^{\top}=0$ for all $s \in \Gamma(P \oplus Z)$. 

By Lemma \ref{lemma:condition2}, $H^0(P^* \otimes (P \oplus Z)^*)=\{0\}$.

We have shown that the hypotheses of Proposition \ref{proposition:holomorphic} are satisfied. Therefore $\Sigma_k$, and hence $\Sigma$, is holomorphic with respect to an orthogonal complex structure on $\mathbb{R}^n$
\end{proof}

\section{Stability in PIC manifolds} \label{pic}
We now consider compact stable surfaces in manifolds with positive curvature on isotropic two-planes (PIC). Let $(N,g)$
be a Riemannian manifold of dimension $n\geq 4$ and recall from \cite{MM} that we may complexify the tangent space and
extend the curvature $(0,4)$ tensor complex multilinearly to define the complex sectional curvatures $K(\Pi)$ for a complex
two dimensional subspace of the complexified tangent space
\[ K(\Pi)=\frac{R(X,Y,\bar{X},\bar{Y})}{\|X\wedge Y\|^2}.
\]
We may then restrict to planes $\Pi$ which are isotropic in the sense that $(X,X)=0$ for all $X\in \Pi$ where we denote by
$(\cdot,\cdot)$ the complex linear pairing which extends the metric $g$. We say that a manifold is PIC if $K(\Pi))>0$ for 
all isotropic two-planes $\Pi$ and for a number $\kappa>0$ we say that $N$ is $\kappa$-PIC if $K(\Pi)\geq \kappa$ for
all isotropic two-planes. To write the stability condition assume that $M$ is a compact Riemann surface and $f:M\to N$
is a conformal parametrization of the minimal surface. The stability condition is then (see \cite{MM})
\[ \int_M [R(s,f_z,\bar{s},\bar{f_z})+\|\nabla_z^\top s\|^2]\ dxdy\leq \int_M \|\nabla_{\bar{z}}^\perp s\|^2\ dxdy
\]
where $s$ is an isotropic section of the complexified normal bundle, $f_z$ denotes the image of $\partial/\partial z$ under $f$, 
$\nabla$ is the induced connection on the pullback
of the tangent bundle of $N$ by $f$, and $z=x+iy$ is a local complex coordinate on $M$. If $N$ is $\kappa$-PIC we may
throw away the second term on the left to obtain the inequality 
\[ \kappa\int_M\|s\|^2\ da\leq \int_M \|\nabla_{\bar{\epsilon}}^\perp s\|^2\ da
\]
where $\epsilon=f_z/\|f_z\|$ and the area form $da=2\|f_z\|^2\ dxdy$.

We want to construct holomorphic or almost holomorphic isotropic sections, so we first suppose that we have a line
bundle $L\subset E$ where $E=T_{\mathbb C}^\perp M$ with its complex structure determined by the normal connection.
We may restrict the complex linear pairing $(\cdot,\cdot)$ to $L$, and we observe that the pairing is either identically $0$
in which case $L$ is isotropic or it is nonzero (hence nondegenerate) away from a finite number of points of $M$. This is
because the paring is holomorphic, so if we choose a local nonzero holomorphic section $s$ of $L$ then $(s,s)$ is a
holomorphic function and thus has isolated zeroes or vanishes identically.
\begin{lemma} \label{pairing} If $L$ has positive degree or is a line bundle of degree $0$ which is not isomorphic to its dual, then $L$ is isotropic. If $L_1$ and $L_2$ are  degree $0$ line bundles with $L_1$ not equivalent to $L_2^*$, then they are orthogonal with respect to the pairing.
If $L$ is a line bundle such that the paring is nondegenerate on $L$ then there is a holomorphic splitting $E=L\oplus E_1$
such that the pairing is also nondegenerate on $E_1$.
\end{lemma}
\begin{proof} First assume that $L$ has non-negative degree $d$ and the pairing is not identically zero. We may choose a 
meromorphic section $s$ of $L$ whose zeroes and poles are away from the zeroes of the pairing. The meromorphic function
$(s,s)$ then has zeroes and poles of twice the order of those of $s$ and it also has zeroes at the zeroes of the pairing. Thus
if $d>0$ the function has more zeroes then poles and is identically zero so that $L$ is isotropic. If $d=0$ then it follows that
the pairing is either identically zero or nowhere zero. 

If $d=0$ and the pairing is nowhere zero, then it follows that the pairing is nondegenerate. The pairing then gives an isomorphism
of $L$ with its dual. If $L_1$ and $L_2$ are degree $0$ line bundles which are not orthogonal with respect to the pairing then
the pairing must be nondegenerate on $L_1\times L_2$ since the pairing of meromorphic sections is a meromorphic function
which has the same number of zeroes as poles. Thus the pairing defines an isomorphism of $L_1$ with $L_2^*$.

If the pairing is nondegenerate on $L$ we can let $E_1$ be the orthogonal complement bundle of $L$
which is then holomorphic and nondegenerate. 
\end{proof}
We remark that there are $4$ distinct self dual line bundles over a torus corresponding to the order $2$ points in the torus (as an abelian
group) generated by $1,\tau$; that is, the points $0,\ 1/2,\ \tau/2,\ 1/2(1+\tau)$. This is because we associate with a degree $0$ line 
bundle the divisor $\{p\}-\{0\}$ where $p$ is a point of the torus. This correspondence gives a group isomorphism from the degree
$0$ line bundles under tensor product with the torus. Notice that all of these line bundles become trivial when lifted to the four fold
covering corresponding to the sub-lattice spanned by $2$ and $2\tau$.

We now restrict to the case at hand when $M$ is a torus and $f:M\to N$ is stable while $N$ is $\kappa$-PIC for 
some $\kappa>0$. We begin with the following proposition.
\begin{proposition} \label{pic_torus} If $M$ is a stable torus in a PIC manifold $N^n$ with $n\geq 4$, then if $n=4$ or $n\geq 7$, 
the  complexified normal bundle $E$ contains a degree $0$ isotropic line sub-bundle. For $n=5,6$, either $E$ contains a degree $0$
isotropic line sub-bundle or the normal bundle of $M$ splits into a direct sum of degree $0$ line bundles which are orthogonal with respect to
the complex linear pairing. In particular each line bundle is self dual. 
\end{proposition} 

\begin{proof} In general for any surface $M$ we have $E=P\oplus Z\oplus P^*$ where $P$ is a direct sum of indecomposable
bundles of positive degree. First we show that $P$ must be trivial. To see this suppose we had an indecomposable
bundle $P_1$ of positive degree. 
We see from the Riemann-Roch theorem that $P_1$ has a holomorphic section $s$. 
We see that $s$ is isotropic because otherwise
it would span a degree $0$ line bundle which splits, and this is not possible since $P_1$ is indecomposable. By stability of
$M$ there can be no holomorphic isotropic section, so we conclude that $P$ is trivial.

Therefore we have shown that $E$ is a direct sum of indecomposable bundles of degree zero. If there is a summand $E_1$
of rank $d_1>1$, then we know from \cite{A} that $E_1$ has a unique line sub-bundle $L$ of degree $0$. By Lemma \ref{pairing}
we see that the complex linear pairing must be trivial on $L$ since $E_1$ is indecomposable. Therefore $L$ is a degree $0$
isotropic sub-bundle. 

It remains to consider the case when $E$ is a direct sum of degree $0$ line bundles $L_1,\ldots, L_r$ with the pairing being
nondegenerate on each $L_j$ and with distinct bundles being orthogonal with respect to the pairing. Since each $L_j$ is isomorphic to its
dual we see that $L_j$ has a flat $U(1)$ structure with holonomy in $\{1,-1\}$. If we choose a basis $\{1,\tau\}$ for our lattice,
we see that there are only four such bundles. Therefore if $n\geq 7$ then $r=n-2\geq 5$, and there must be a repeated bundle.
By renumbering we assume that $L_1$ and $L_2$ are isomorphic. Let $\hat{M}$ be the four fold covering corresponding to
the sub-lattice spanned by $2$ and $2\tau$. The direct sum of the lifted bundles is then trivial and the pairing is nondegenerate, so
we can find sections $\hat{s_1}$ and $\hat{s_2}$ of the direct sum of the lifted bundles which are orthonormal. Thus we have 
$(\hat{s_j},\hat{s_k})=\delta_{jk}$ for $j,k=1,2$. The section 
$\hat{s}=\hat{s_1}+i\hat{s_2}$ is then isotropic and is a section of the lift of a degree $0$ isotropic sub-bundle $L$ of $L_1\oplus L_2$
over $M$.

For $n=4$, if we assumed that $N$ is oriented, there would be a natural orientation and metric on the two dimensional normal spaces.
These give a natural parallel complex structure given by rotation by $\pi/2$ on the normal bundle and it gives a parallel decomposition of $E$
into the $(1,0)$ and $(0,1)$ summands which are isotropic line sub-bundles. Neither can have positive degree or else it would
have  a section contradicting the PIC assumption. Therefore both are degree $0$ isotropic line sub-bundles. Without the orientation,
we can still complexify the normal bundle and split it into a direct sum of line bundles. If neither line bundle is isotropic, then since
one is the dual of the other it follows that they are isomorphic. As above this gives us a degree $0$ isotropic line sub-bundle of $E$.

The remaining cases are $n=5,6$ and $E$ is a direct sum of degree $0$ line bundles on which the complex linear pairing
is nondegenerate and such that no two line bundles are isomorphic. By Lemma \ref{pairing} it follows that the line bundles
are orthogonal with respect to the pairing. This completes the proof of Proposition \ref{pic_torus}.
\end{proof}

We now prove the main theorem of this section. Given a branched minimal immersion $f:M\to N$ where $M$ is a genus $1$ Riemann
surface, we consider the pulled back metric on $M$ and define the {\it systole} $R$ to be the length of the shortest homotopically
nontrivial closed curve on $M$. 
\begin{theorem} \label{systole} Suppose $N^n$ ($n\geq 4$) is a $\kappa$-PIC manifold for some $\kappa>0$ and suppose
$f:M\to N$ is a stable conformal branched minimal immersion of genus $1$. There is an absolute constant $C>0$ so that $R\leq C/\sqrt{\kappa}$.
In the general case for $n=4$ or $n\geq 7$ we can take $C=2\pi/\sqrt{3}$, while in those cases with $n=5,6$ in which there is
no degree $0$ isotropic line sub-bundle of $E$ we can take $C=2(18+\pi)/\sqrt{3}$.
\end{theorem}

\begin{proof} We first consider the main case from Proposition \ref{pic_torus} when $E$ has an isotropic degree $0$ line sub-bundle $L$. 
The exceptional cases will be handled separately. We then use the same
idea as in Lemma \ref{flatline} except we use the distance in the pulled back metric instead of the flat metric on $M$. We also note that
we can take our lattice to be generated by $1$ and $\tau$ where $\tau=\tau_1+i\tau_2$ is in the upper half plane with $|\tau|\geq 1$ and 
$|\tau_1|\leq 1/2$ since any torus can be represented by such a lattice. The line bundle $L$ is trivial on the universal covering $\mathbb C$ 
and the holonomy is in $U(1)$. Let $s_0$ be a global nonzero section on $\mathbb C$ and assume that $\phi,\theta\in (-\pi, \pi]$
so that $s_0(z+1)=e^{i\phi}s_0(z)$ and $s_0(z+\tau)=e^{i\theta}s_0(z)$. As in Lemma \ref{flatline} we make the change of 
coordinates $\xi=x-y\frac{\tau_1}{\tau_2}$ and $\eta=\frac{y}{\tau_2}$. Consider the fundamental domain $F=\{(\xi,\eta):\ 0\leq \xi,\eta<1\}$.
We have $s_0(1,\eta)=e^{i\phi}s_0(0,\eta)$ and $s_0(\xi,1)=e^{i\theta}s_0(\xi,0)$ since $z=\xi+\eta\tau$. 

We let $d(z,w)$ denote the induced distance with respect the pulled back metric of $N$ via the map $f$ lifted to $\mathbb C$, and observe from the definition of the systole $R$ we have $d(z,z+1)\geq R$ and $d(z,z+\tau)\geq R$. We now define a function $\delta(z,w)$ by
\[ \delta(z,w)=\min\{d(z,w),R\}.
\]
Thus we have $\delta(z,z+1)=\delta(z,z+\tau)=R$, and $\delta$ is Lipschitz with gradient in the pulled back metric bounded by $1$ in
both arguments. We now define an isotropic section of $E$ over $M$ by
\[ 
     s(\xi,\eta)=e^{-i \left( \frac{\delta((0,0),(\xi,0))}{R}\phi- \frac{\delta((0,0),(0,\eta))}{R}\theta \right)}s_0(\xi,\eta)
\]
for $0\leq \xi,\eta\leq 1$. We see that 
\[
    s(1,\eta)=e^{-i\phi}e^{-i\frac{\delta((0,0),(0,\eta))}{R}\theta}s_0(1,\eta)=s(0,\eta)
\] 
and similarly we
have $s(\xi,1)=s(\xi,0)$. Therefore $s$ defines a Lipschitz isotropic section of $E$ over $M$. Using this in the stability inequality
and using the $\kappa$-PIC condition we have
\[ \kappa\int_M\|s\|^2\ da\leq \int_M \|\nabla_{\bar{z}}s\|^2\ dxdy.
\]
If we write the induced metric as $\lambda^2(dx^2+dy^2)$, we then have 
\[
        \|\nabla_{\bar{z}}\delta((0,0),(\xi,0))\|
        \leq \lambda \left|\frac{\partial\xi}{\partial\bar{z}}\right|
        =\lambda \frac{\sqrt{1+\tau_1^2/\tau_2^2}}{2}\leq \frac{\lambda}{\sqrt{3}}
\]
since $|\tau_1/\tau_2|\leq \sqrt{1/3}$. Similarly we have 
\[ 
       |\nabla_{\bar{z}} \rho((0,0),(0,\eta))\|
       \leq \lambda \left|\frac{\partial\eta}{\partial\bar{z}}\right|
       =\frac{\lambda}{2\tau_2}\leq \frac{\lambda}{\sqrt{3}}.
\]
Taken together these imply that
\[ \|\nabla_{\bar{z}}s\|^2\leq \left(\frac{2\pi}{\sqrt{3}R}\right)^2\|s\|^2\lambda^2.
\]
Since $da=\lambda^2\ dxdy$, it follows from stability that
\[  \kappa\int_M\|s\|^2\ da\leq \left(\frac{2\pi}{\sqrt{3}R}\right)^2 \int_M \|s\|^2\ da,
\]
and therefore $R\leq C/\sqrt{\kappa}$ with $C=2\pi/\sqrt{3}$.

We now deal with the exceptional cases with $n=5,6$ and $E$ splitting into a direct sum of distinct self-dual line bundles which
are orthogonal with respect to the complex linear pairing. In these cases it can happen that there is no degree $0$ isotropic line
sub-bundle. We first begin with the case $n=5$ so that $E$ has rank $3$. Then we have $E=L_1\oplus L_2\oplus L_3$ with each
$L_j$ having holonomy in $\{-1,1\}\subset U(1)$. The holonomy is generated by $\rho(1)$ and $\rho(\tau)$ which take values
in $\{-1,1\}^3$. Since there is no isotropic sub-bundle we cannot have the same holonomy in two slots. Therefore $\rho(1)$ 
and $\rho(\tau)$ are both nontrivial. By reordering the bundles we may assume that $\rho(1)=(a,a,b)$ where $a,b\in\{-1,1\}$. We
consider the two fold covering $\hat{M}$ of $M$ corresponding to the sub-lattice $1,2\tau$. We see that the lift of $E$ to $\hat{M}$
has trivial holonomy in the vertical period $2\tau$ since $\rho(2\tau)=(1,1,1)$. Therefore the lifts $\hat{L_1},\ \hat{L_2}$ of $L_1,\ L_2$
are isomorphic. It follows that there is a degree $0$ isotropic sub-bunde of $\hat{L_1}\oplus \hat{L_2}$, and a holomorphic isotropic
section $s_0$ on the fundamental domain $\hat{F}=\{(\xi,\eta):\ 0\leq \xi\leq 1,\ 0 \leq\eta\leq 2\}$ satisfying $s_0(1,\eta)=as_0(0,\eta)$
and $s_0(\xi,2)=s_0(\xi,0)$ where $a\in\{-1,1\}$. We see that if $n=6$ the exceptional case is when $E$ is the orthogonal direct sum of the
four self dual line bundles. In this case we can make the same construction, first reordering so that $\rho(1)$ has the same entry
in the first two slots and then taking the same two fold covering of $M$ with an isotropic section as above.  

We now work on $\hat{M}$. Since we don't know that it is stable, we must make a more complicated construction. We first remove the
horizontal period of $s_0$, by setting $s_1=s_0$ if $a=1$, and 
\[ 
     s_1(\xi,\eta)=e^{-i\delta((0,0),(\xi,0))\pi/R}s_0(\xi,\eta).
\]
We then have $s_1(1,\eta)=s_1(0,\eta)$, and from the bounds above we have
\[ 
      \|\nabla_{\bar{z}}s_1\|\leq  \frac{\pi}{\sqrt{3}R}\|s_1\|\lambda.
\]
We can now go to the quotient $\hat{M_1}$ of $\mathbb R^2$ gotten by identifying $(\xi+1,\eta)$ with $(\xi,\eta)$. Thus $\hat{M_1}$
is $\mathbb S^1\times \mathbb R$ and $s_1$ is an isotropic section over $\hat{M_1}$.
We now divide $\hat{M}$ into $4$ disjoint sets which project $1$-$1$ to $M$, and localize the section $s_1$ in such a set. To define these
sets we use the vertical distance functions on $\hat{M_1}$, $d_t(\xi,\eta)=d((\xi,\eta),(\xi,t))$ where $\xi$ is defined mod $\mathbb Z$. 
From the systole definition we have $d_t(\xi,t+1)\geq R$. 
We now define sets $U_j$ for integers $j$ by 
\[ U_j=\{(\xi,\eta):\ d_j(\xi,\eta)\leq R/3\}
\]
Note from the $1$-periodicity of the metric we have $U_{j+1}=U_j+(0,1)$ and the sets have the same projection $U$ to $M$. We now define 
$V$ to be the complement of $U$ in $M$ and we observe that the lift $\hat{V}$ of $V$ to $\hat{M_1}$ separates into a disjoint union of $V_j$
with $V_j=\hat{V}\cap \{(\xi,\eta):\ j-1\leq \eta\leq j\}$. Note that lift of $U,V$ to $\hat{M}$ defines a decomposition into four sets with disjoint
interior. 

For each $j$ we define the integrals $I_j,J_j$ by $I_j=\int_{U_j}\|s_1\|^2\ da$ and $J_j=\int_{V_j}\|s_1\|^2\ da$. Since $s_0$ is periodic
with period $2$ (but not period $1$) we have $I_{j+2}=I_j$ and $J_{j+2}=J_j$. The four values $I_0,I_1,J_0,J_1$ contain all of the
distinct values of the integral. We choose a largest one of these and localize $s_1$ to a neighborhood of that set. The argument is
the same if the largest value is one of the $U$'s or if it is one of the $V$'s, so we deal with the two cases when the maximum is $I_1$
and when it is $J_0$.

If the maximum occurs for $I_1$ we construct a cut-off function $\varphi$ which is $1$ on $U_1$ and zero outside $d_1\geq R/2$.
Specifically we define $\varphi=1$ in $U_1$, $\varphi=0$ at points where $d_1\geq R/2$, and when $R/3\leq d_1\leq R/2$
\[ \varphi=3-\frac{6}{R}d_1.
\]
Note that the support of $\varphi$ projects $1$-$1$ to $M$ because two points $(\xi_1,\eta_1)$ and $(\xi_2,\eta_2)$ would have
$\xi_1=\xi_2$ and $\eta_1$ and $\eta_2$ would differ by an integer, so it would follow that $d((\xi,\eta_1),(\xi,\eta_2))\geq R$,
but for any two points in the support of $\varphi$ we have by the triangle inequality 
$d((\xi,\eta_1),(\xi,\eta_2))\leq d_1(\xi,\eta_1)+d_1(\xi,\eta_2)<R$. Thus we see that the section $\varphi s_1$ defines an isotropic
section of $E$ over $M$ so we can use it in the stability inequality to obtain
\[ \kappa I_1\leq \int _M\|\nabla_{\bar{z}}\varphi s_1\|^2\ dxdy.
\]
We can estimate the term on the right
\[ 
      \|\nabla_{\bar{z}}\varphi s_1\|
      \leq \|\nabla_{\bar{z}}\varphi\|\|s_1\|+\varphi\|\nabla_{\bar{z}}s_1\| 
      \leq (\|\nabla_{\bar{z}}\varphi\|+\frac{\pi}{\sqrt{3}R}\lambda)\|s_1\|
\]
where we have used our previous bound. Now we can estimate 
\[ 
      \|\nabla_{\bar{z}}\varphi\|
      \leq \frac{6}{R}\|\nabla_{\bar{z}}d_1\|
      \leq \frac{6}{R}(2\|\nabla_{\bar{z}}\xi\|+\|\nabla_{\bar{z}}\eta\|).
\]
From our previous bounds this implies
\[  \|\nabla_{\bar{z}}\varphi\|\leq \frac{6}{R}\sqrt{3}\lambda.
\] 
Now we have the bound on the integral
\[ 
       \int _M\|\nabla_{\bar{z}}\varphi s_1\|^2\ dxdy
       =\int _{\hat{M}}\|\nabla_{\bar{z}}\varphi s_1\|^2\ dxdy
       \leq \left(\frac{18+\pi}{\sqrt{3}}\right)^2  \int_{\hat{M}}\|s_1\|^2\ da.
\]
Since $I_1$ was the largest of the integrals it follows that the integral over $\hat{M}$ is at most $4$ times $I_1$, so we have
\[ 
      \kappa I_1\leq \frac{4}{R^2} \left(\frac{18+\pi}{\sqrt{3}}\right)^2 I_1,
\]
and so $R\leq  2 (18+\pi) / ( \sqrt{3} \sqrt{\kappa} )$ as claimed.

Finally we consider the case in which $J_0$ is the largest of the four integrals. Recall that the set $V_0$ is contained in the set with 
$0<\eta<1$ and is the set of points for which both $d_0$ and $d_1$ are at least $R/3$. Recall that $V_0$ can be describes 
as the set of points with $0<\eta<1$ with $\min\{d_0,d_1\}\geq R/3$. Thus we define $\varphi$ as $\varphi=1$ on $V_0$, and
$\varphi=3/R\min\{d_0,d_1\}$ otherwise for $\eta\leq 1$. We then observe that the set of points with $\varphi>0$ is contained in
the set with $0<\eta<1$, and so projects $1$-$1$ into $M$. Thus we can follow the argument above and use $\varphi s_1$ as
a variation. The only difference is that the bound on $\|\nabla_{\bar{z}}\|$ is smaller by a factor of $2$. Therefore we get the bound
$R\leq  (18+\pi) / ( \sqrt{3} \sqrt{\kappa} )$ in this case. This completes the proof of Theorem \ref{systole}.

\end{proof}

We now give an application to the geometry and topology of PIC manifolds which sharpens and generalizes the theorem of \cite{F3}. Assume 
that $N$ is a
compact Riemannian manifold. Given a subgroup $G$ of $\pi_1(N)$ we can define the {\it systole} of $G$ to be the shortest length of
a any curve which is freely homotopic to an element of $G$. Thus if $G=\pi_1(N)$, then the systole is just that of $N$. We now state the
main application to PIC manifolds.
\begin{theorem} \label{pic-manifolds} Assume that $G$ is an abelian non-cyclic subgroup of $\pi_1(N)$ and that $N$ is a compact $\kappa$-PIC
manifold. The systole $R$ of $G$ is bounded by $C/\sqrt{\kappa}$ for the same constant $C$ as in Theorem \ref{systole}.
\end{theorem}
\begin{remark} Note that the systole of $N$ can be arbitrarily large for a $\kappa$-PIC manifold as illustrated by the product of a long
circle with a constant curvature sphere of dimension at least $3$.
\end{remark}
\begin{proof} The theorem follows from the existence of an area minimizing map of a torus among maps which are surjective
from $\pi_1(M)$ to $G$. This existence result is discussed in Example 1.4 of \cite{F3}. We then observe that the systole of the
torus is at least as large as the systole of $G$ because that of the torus is realized by the length of a simple closed geodesic,
and the image of any simple closed curve is freely homotopic to a nontrivial element of $G$. The bound of Theorem \ref{systole}
then gives the desired bound.
\end{proof}

Finally we show that the main theorem of \cite{F3} follows from Theorem \ref{pic-manifolds}.
\begin{theorem} (Fraser \cite{F3}) \label{free-abelian} Suppose $N$ is a compact PIC manifold. Then $\pi_1(N)$ cannot contain a free abelian subgroup of rank
greater than $1$.
\end{theorem}
\begin{proof} Since $N$ is compact and PIC, it is $\kappa$-PIC for some $\kappa>0$. Suppose $G$ is a free abelian subgroup of
rank $2$ contained in $\pi_1(N)$. Thus $G$ is generated by a pair of elements $\gamma_1,\gamma_2$ which have infinite order
and which commute in $\pi_1(N)$. Thus any element of $G$ can be written as $\gamma_1^p\gamma_2^q$ for integers $p,q$. For
any integer $k\geq 2$ we consider the subgroup $G_k$ of $G$ generated by $\gamma_1^k,\gamma_2^k$. If we let $R_k$ be the
systole of $G_k$, we show that $\lim_{k\to\infty}R_k=\infty$. To see this we note that since the fundamental group acts properly 
discontinuously on the universal cover of $N$, we have a sequence of group elements $\gamma_j$ which go to infinity in the
group with respect the word metric. It follows that the translation distance (systole) of $\gamma_j$ goes to infinity. For each $k$
the systole of $G_k$ is realized by an element $\gamma_k$, and all nontrivial elements of $G_k$ go to infinity with $k$. Therefore
it follows that $R_k\to\infty$ and thus for $k$ large enough we have $R_k>C/\sqrt{\kappa}$ in violation of Theorem \ref{pic-manifolds}
since each $G_k$ is an abelian group which is not cyclic.

\end{proof}
     
\bibliographystyle{plain}

\begin{thebibliography}{BFNT}

\bibitem{AMP} C. Arezzo, M. Micallef, G.~P. Pirola, Stable minimal surfaces of finite total curvature,
                 {\em Comm. Anal. Geom.} {\bf 10} (2002), no. 1 11--22.
\bibitem{A} M.~F. Atiyah, Vector bundles over an elliptic curve, {\em Proc. London Math. Soc. (3)} {\bf  7} (1957), 
               414--452.
\bibitem{B}  S. Brendle, Ricci flow with surgery on manifolds with positive isotropic curvature, 
              {\em Ann. of Math. (2)} {\bf 190} (2019), no. 2, 465--559.    
\bibitem{CZ} B.-L. Chen and X.-P. Zhu, Ricci flow with surgery on four-manifolds with positive isotropic curvature, 
              {\em J. Differential Geom.} {\bf 74} no. 2 (2006), 177--264.                                     
\bibitem{F3} A. Fraser, Fundamental groups of manifolds of positive isotropic curvature, 
             {\em Ann. of Math. (2)} {\bf 158} (2003), no. 1, 345--354.        
\bibitem{Gu}  N.~P. Gushel', Very ample divisors on projective bundles over an elliptic curve,
              {\em Mat. Zametki} {\bf 47} (1990), no. 6, 15--22, 158; translation in {\em Math. Notes} {\bf 47} (1990), 
              no. 5-6, 534--539.
\bibitem{H} R.~S. Hamilton, Four-manifolds with positive isotropic curvature, {\em Comm. Anal. Geom.} {\bf 5} 
                  no. 1 (1997), 1--92.              
\bibitem{KM} J.-L. Koszul, B. Malgrange, Sur certaines structures fibr\'ees complexes, {\em Arch. Math. (Basel)} 
              {\bf 9} (1958), 102--109.              
\bibitem{LS} B. Lawson, J. Simons, On stable currents and their application to global problems in real and 
                complex geometry, {\em Ann. of Math. (2)} {\bf 98} (1973), 427--450.                    
\bibitem{Mi} M. Micallef, Stable minimal surfaces in Euclidean space,
                         {\em J. Differential Geom.} {\bf 19} (1984), no. 1, 57--84.
\bibitem{MM} M. Micallef, J.~D. Moore, Minimal two-spheres and the topology of manifolds with positive
               curvature on totally isotropic two-planes, {\em Ann. of Math. (2)} {\bf 127} (1988), no. 1, 199--227.  
\bibitem{O} R. Osserman, Global properties of minimal surfaces in $E^3$ and $E^n$, {\em Ann. of Math. (2)}
               {\bf 80} (1964), 340--364.           
\bibitem{SiY} Y.-T. Siu, S.-T. Yau, Compact K\"{a}hler manifolds of positive bisectional curvature, 
               {\em Invent. Math.} {\bf 59} (1980), no. 2, 189--204.    
                                    
\end{thebibliography}

\end{document}